\documentclass[12pt,leqno]{amsart}
\usepackage{amsmath,amsfonts,amssymb}
\usepackage[margin=0.8in]{geometry}

\renewcommand{\leq}{\leqslant}
\renewcommand{\le}{\leqslant}
\renewcommand{\geq}{\geqslant}
\renewcommand{\ge}{\geqslant}

\newcommand{\R}{\mathbb{R}}
\newcommand{\N}{\mathbb{N}}

\newtheorem{thm}{Theorem}[section]
\newtheorem{lem}[thm]{Lemma}
\newtheorem{cor}[thm]{Corollary}
\newtheorem{prop}[thm]{Proposition}

\newcommand{\rl}{\mathbb R}

\newcommand{\ep}{\varepsilon}
\newcommand{\ld}{\lambda}
\newcommand{\mm}{{\mathcal M}}

\usepackage{color}

\title[Convex sets evolving
by volume preserving
fractional mean curvature flows]{Convex sets evolving \\
by volume preserving
fractional mean curvature flows}
\author[Eleonora Cinti]{Eleonora Cinti}
\address{E.C., Dipartimento di Matematica, Universit\`a degli Studi di Bologna,
Piazza di Porta San Donato 5, 40126 Bologna, Italy}
\email{eleonora.cinti5@unibo.it}

\author[Carlo Sinestrari]{Carlo Sinestrari}
\address{C.S., Dipartimento di Ingegneria Civile e Ingegneria Informatica, 
Universit\`a di Roma Tor Vergata,
Via del Politecnico,
00133 Rome, Italy}  \email{sinestra@mat.uniroma2.it}
\author[Enrico Valdinoci]{Enrico Valdinoci}
\address{E.V., 
Department of Mathematics and Statistics,
University of Western Australia,
35~Stirling Highway, WA 6009 Crawley, Australia. }
\email{enrico.valdinoci@uwa.edu.au}

\thanks{E.C. was supported
 by MINECO grant MTM2014-52402-C3-1-P
and is part of the Catalan research group 2014 SGR 1083.
E.C. and  C.S. were supported by the group 
GNAMPA of INdAM Istituto Nazionale di Alta Matematica.
E.C. and E.V. were supported by the European Research
Council Starting Grant 
``EPSILON'' 
Elliptic Pde's and Symmetry of Interfaces and Layers for Odd Nonlinearities
n. 277749.
E.V. was supported by the Australian Research Council
Discovery Project ``NEW''
Nonlocal Equations at Work n. DP170104880. The authors are members of INdAM/GNAMPA}



\begin{document}

\begin{abstract}
We consider the volume preserving
geometric evolution of the boundary of a set
under fractional mean curvature. We show that
smooth convex solutions maintain their fractional curvatures bounded
for all times, and the long time asymptotics
approach round spheres.
The proofs are based on apriori estimates
on the inner and outer radii of the solutions.
\end{abstract}

\maketitle

\section{Introduction}

Let $E_0 \subset \rl^n$ be a smooth compact convex set, and let $\mm_0=\partial E_0$. For a fixed $s \in (0,1)$, we consider the evolution of $\mm_0$ by volume preserving fractional mean curvature flow, that is
the family of immersions $F:\mm_0 \times [0,T) \to \rl^n$ which satisfies 
\begin{equation}\label{flow}
\left\{ \begin{array}{ll}
\partial_t F(p,t) = \left[- H_s(p,t) +h(t) \right] \, \nu(p,t), \quad & p \in \mm_0, t \geq 0 \medskip \\
F(p,0)=p & p \in \mm_0.
\end{array}\right.
\end{equation}

Here $H_s(p,t)$ and $\nu(p,t)$ denote respectively the fractional mean curvature of order $s$  and the normal vector of the hypersurface $\mm_t:=F(\mm_0,t)$ at the point $F(p,t)$, while the function $h(t)$ is defined as
\begin{equation}\label{volpr}
h(t)=\frac{1}{|\mm_t|}\int_{\mm_t} H_s(x) d\mu,
\end{equation}
where $d\mu$ denotes the surface measure on $\mathcal M_t$. With this choice of $h(t)$, the set $E_t$ enclosed by $\mm_t$ has constant volume. An  interesting feature of this flow is that the fractional $s$-perimeter of $E_t$ is decreasing, and the monotonicity is strict unless $E_t$ is a sphere.

Fractional (or nonlocal) mean curvature was first defined by Caffarelli, Roquejoffre and Savin in \cite{CRSavin}. It arises naturally when performing the first variation of the fractional perimeter, a nonlocal notion of perimeter introduced in the same paper. We will recall the definitions of these quantities in \S 2. Minimizers of the fractional perimeter are usually called nonlocal minimal sets, and their boundaries nonlocal minimal surfaces. Fractional perimeter and mean curvature have also found  application in other contexts, such as image reconstruction and nonlocal capillarity models, see e.g. \cite{BS2015, MV}. 

Nonlocal minimal surfaces have attracted the interest of many researchers in the last years. One of the main issues is the study of their regularity and the classification of nonlocal minimal cones: many results have been obtained, see \cite{CRSavin, CV, BFV, SV, CSV, CCS}, which exhibit interesting analogies and differences with respect to the classical case. Among
the important differences, we mention in particular the fact
that fractional minimal surfaces can stick at the boundary of (even smooth
and convex) domains, and occupy all the domain for small values
of the fractional parameter, see~\cite{MR3596708}: these features are in sharp contrast
with the classical case and they reveal the important
role of the contributions coming from infinity in the geometric
displacements of nonlocal minimal surfaces.

A related topic of investigation consists in the study of sets which are stationary for the fractional perimeter, i.e. sets having vanishing nonlocal mean curvature. This is a  weaker notion than minimality, and some examples are helicoids and a nonlocal version of catenoids, see e.g. \cite{DDPW, CDDP}. Sets with constant nonlocal mean curvature, such as Delaunay-type surfaces, have been studied in \cite{CFSW, MR3485130, CFW1,CFW2}. In addition, in \cite{CFSW, Ciraolo-etal} it was proved an analogue of the Alexandrov Theorem in the nonlocal setting, which will be crucial for our purposes: any, regular enough, bounded set having constant fractional mean curvature is necessarily a ball.

Before introducing our results, let us recall some properties of the {\em classical} mean curvature flow, where the speed of the hypersurface is given by the usual mean curvature. This flow has been widely studied in the last decades, both for its geometric interest and for its relevance in physical models describing the dynamics of interfaces. The equation satisfied by the immersion is a parabolic PDE, and smooth solutions exist locally; however, they can become singular in finite time due to curvature blowup. For this reason, various notions of weak solutions have been introduced during the years, which allow to continue the evolution after the formation of singularities, see e.g. \cite{CGG,ES1}.

An important feature of classical mean curvature flow is that, roughly speaking, it deforms general hypersurfaces into some canonical profiles, possibly after rescaling near the singularities. Such a behaviour is related to the diffusive character of the flow and is of great interest for geometric applications. The first result on asymptotic convergence was obtained by Huisken in \cite{Hu84} in the $h(t) \equiv 0$ case. He  proved that
{\em convex hypersurfaces remain smooth up to a finite maximal time at which they shrink to a point, and that they converge to a round sphere after rescaling}. Shortly afterwards, in \cite{Hu87}, he obtained an analogous result for the {\em
volume preserving flow}: in this case, the solution exists for all times and converges to a sphere as $ t \to +\infty$. In later years, many researchers have studied the convergence to a sphere for other kinds of geometric flows, with a speed driven by more general functions of the (classical) principal curvatures, see e.g. \cite{AMZ,AWei}. As a possible application of these results, we point out that the convergence to a sphere along a suitable flow can be used to obtain generalizations or alternative proofs of classical geometric inequalities, such as the isoperimetric inequality, or inequalities in convex analysis like the ones by Minkowski or Alexandrov-Fenchel, see e.g. \cite{McCoy2005, Schulze, GLi, ACWei}. 


By contrast, the study of fractional mean curvature flow has started only recently and very few results are known. The existence and uniqueness of weak solutions in the viscosity sense for the flow in the $h(t) \equiv 0$ case has been obtained by various authors with different approaches \cite{Imbert, CSou, CMPonsiglione}. In particular, in \cite{CSou} Caffarelli and Souganidis proved convergence to motion by fractional mean curvature  of a threshold dynamics scheme. More recently, in \cite{CNRuffini},  Chambolle, Novaga and Ruffini, have extended the results in \cite{CSou} to the anisotropic case and to the presence of an external driving force (that is $h(t) \neq 0$) and have proved that the scheme preserves convexity, and, as a consequence, also the limit geometric evolution is convexity preserving. It is commonly expected that a local existence theorem for smooth solutions should hold as well, but no proof of it is available until now to our knowledge. Some qualitative properties of smooth solutions have been analyzed in \cite{SaezV}, while the formation of neckpinch singularities has been studied in \cite{CintiSV}. The occurrence of fattening for the fractional mean curvature flow and its generalizations has been studied in \cite{CesaroniDNV}.

The aim of this paper is to study the
{\em convergence to a sphere of the solutions of the nonlocal flow \eqref{flow} with convex initial data}.
This can be regarded as the first attempt to investigate the asymptotic behaviour of solutions to fractional flows, in a similar spirit to the above mentioned works in the classical case. Our main results are some apriori estimates on smooth solutions, which give a uniform control on the geometry of the evolving surfaces, and establish that the fractional curvature remains uniformly bounded along the flow. As a consequence, we can show that any smooth solution, satisfying suitable regularity assumptions, exists for all times and converges to a sphere. The method is inspired by the one of \cite{Andrews-aniso, Sinestrari-CalcVar-2015} in the classical setting and is based on the monotonicity along the flow of the fractional isoperimetric ratio, i.e., the ratio between suitable powers of the fractional perimeter and the enclosed volume. This monotonicity property is peculiar of the volume preserving case, and so the approach used here does not apply when $h(t) \equiv 0$, although we expect that case to exhibit a similar behaviour, at least if $s$ is suitably close to $1$. On the other hand, we include in this paper the treatment of more general flows in the volume preserving setting, with a {\em
nonlinear speed} of the form $\Phi(H_s)$, with $\Phi(\cdot)$ a positive increasing function satisfying suitable structural assumptions.\medskip

Let us describe our results in more detail. For this,
let us denote by~$\underline{\rho_E}$ and~$\overline{\rho_E}$
the inner radius and the outer radius of a set $E \subset \R^n$, namely
\begin{equation}\label{RADIO}
\underline{\rho_E}:=\sup\{ r>0 :  \exists \, x_o\in\R^n, \ B_r(x_o)\subset E \},
\quad
\overline{\rho_E}:=\inf\{ r>0 :  \exists \,  x_o\in\R^n , B_r(x_o)\supset E \}.
\end{equation}
Then our main estimates can be stated as follows:

\begin{thm}\label{main}
Let $E_0$ be a smooth compact convex set of $\R^n$ and let $\mm_0=\partial E_0$. Let $F:\mm_0 \times [0,T) \to \rl^n$, with $0 <T \leq +\infty$, be a solution of \eqref{flow} of class $C^{2,\beta}$ for some $\beta>s$. Then there exist positive constants $0 <R_1 \leq R_2$, $0 < K_1 \leq K_2$, only depending on $E_0$, such that
$$
R_1 \leq \underline{\rho_{E_t}} \leq \overline{\rho_{E_t}} \leq R_2
$$
$$
K_1 \leq H_s(p,t) \leq K_2 \qquad p \in \mm_0,
$$
for all $t \in [0,T)$.
\end{thm} 

As mentioned above, in \cite{CNRuffini} it is proven that the nonlocal mean curvature flow with forcing term ($h(t)\not\equiv 0$) preserves convexity. As a consequence, we know that solutions of problem \eqref{flow} starting from a convex initial datum, stay convex for all times.

The proof of Theorem~\ref{main} relies on a series
of delicate estimates based on a nonlocal analysis
of geometric flavor, which turns out to be significantly different
with respect to the classical case.

Let us describe some intermediate steps in the proof of Theorem \ref{main}, which we believe to have an interest on their own. One of these results, Proposition \ref{bounds},  shows that a bound on the fractional isoperimetric ratio of a convex set implies a bound on the ratio between the outer and inner radii. A similar result was known in the classical case, but the proof in the nonlocal setting is quite different. Another crucial step of our argument is provided by Proposition \ref{hausvogteiplatz}, where we estimate  the fractional mean curvature in terms of another nonlocal quantity, which has some formal analogy with the norm of the second fundamental form in the classical case. However, since there is no fractional analogue of the second fundamental form, as shown in \cite{Abatangelo}, there is no obvious relation as in the classical case. By suitable estimates of the surface integrals involved, we obtain an inequality which suffices for the purposes of this paper; on the other hand, it would be interesting to investigate further these topics and to derive sharper inequalities in the future.

Theorem~\ref{main} easily implies that a solution of \eqref{flow} exists for all times and converges to a sphere as $t \to +\infty$, provided it satisfies suitable regularity and continuation properties. Roughly speaking, we need to know that the solution remains smooth and does not develop singularities as long as the fractional curvature is bounded. More precisely, we assume that there exists a smooth solution of \eqref{flow} satisfying the following property for some $\beta>s$: \medskip

{\bf (R) }
If $H_s$ is bounded on $\mm_t$ for all $t \in [0,T_0)$ for some $T_0 \leq T$, where $T$ is the maximal time of existence, then the $C^{2,\beta}$-norm of $\mm_t$, up to translations, is also bounded for $t \in [0,T)$ by a constant only depending on the supremum of $H_s$. In addition, either $T_0=T=+\infty$, or $T_0<T$.
\medskip

By ``up to translations'', we mean that $\mm_t$ is not assumed to remain in a bounded set of $\R^n$, and that the $C^{2,\beta}$ bound applies after possibly composing the flow with a suitable, time dependent, translation (e.g., the one fixing the barycenter). We give below more comments on the possibility of this behaviour. For solutions satisfying {\bf (R) }, 
the following result holds.

\begin{thm}\label{sphere}
Let $E_0$ be a smooth compact convex set of $\R^n$ and let $\mm_0=\partial E_0$. Let $F:\mm_0 \times [0,T) \to \rl^n$, with $0 <T \leq +\infty$, be a  solution of \eqref{flow} of class $C^{2,\beta}$ for some $\beta>s$ which satisfies property {\bf (R) }. Then $T=+\infty$, and $\mm_t$  converges to a round sphere as $t \to +\infty$ in $C^{2,\beta}$ norm, possibly up to translations.
\end{thm}

We remark that the results of this paper apply to smooth solutions of the flow, whose existence has to be assumed apriori, because the existence theorems available in the current literature only concern viscosity solutions. In this respect, this paper should be regarded as a part of a broader investigation of fractional curvature flows, which should be complemented in the future with the analysis of the local existence of smooth solutions. Regarding assumption {\bf (R)}\,, we observe that it is a natural analogue of some
properties which are well known in the classical case, see e.g. \cite[\S 7-8]{Hu84}, and are consequence of the standard parabolic theory. In view of some recent regularity results on elliptic nonlocal problems, see e.g. \cite{BFV,CL-D1, CL-D2}, we expect that similar results can be obtained for the fractional mean curvature. We plan to investigate these issues in forthcoming work.

As observed above, in Theorem~\ref{sphere} the
convergence to a sphere is in principle only ``up to translations'',
in the sense that the limit set, which is geometrically a sphere, could
keep translating indefinitely.
In the classical case, the possibility of the additional translation is ruled out either as a consequence
of additional estimates on the convergence rate, see e.g. \cite{BP}, 
or by maximum principle techniques based on reflection methods, \cite{MR1386736,McCoy2004,AWei}.
We think that it would be interesting to
understand whether these methods can be extended
to the nonlocal setting.
\medskip

The paper is organized as follows:
\begin{itemize}
\item In Section \ref{S2}, we give some preliminaries and we recall the evolution laws of some geometric quantities associated to $\mathcal M_t$;
\item Section \ref{SE:3} contains our apriori estimates on the inner and outer radii of convex solutions and a lower bound for $H_s$;
\item Section \ref{S4} deals with some integral estimates which allow us
to bound the fractional mean curvature with the nonlocal analogue of the norm of the nonlocal second fundamental form;
\item In Section \ref{S5} we prove our key result which gives an upper bound on the fractional mean curvature;
\item In Section \ref{S6}, we treat the more general case of a flow whose speed is of the form $\Phi(H_s)$, proving an upper bound on the fractional mean curvature;
\item Finally, in Section \ref{S7}, we prove convergence to a sphere in both the standard and the general case.
\end{itemize}

\section{Preliminaries}\label{S2}

Consider a set $E \subset \rl^n$, with boundary $\mm:=\partial E$, and let $s \in (0,1)$. 
Given $x \in \mm$, the {\em fractional mean curvature} of order $s$ of $E$ (equivalently, of $\mm$) at  $x$ is defined by
\begin{equation}\label{defspace}
H_s(x)=s(1-s) \lim_{\ep \to 0^+}\int_{\rl^n \setminus B_\ep(x)} \frac{\tilde \chi_E(y)}{|x-y|^{n+s}} dy ,
\end{equation}
where
$$
\tilde \chi_E(y)= \left \{ \begin{array}{ll}
1 \qquad & \mbox{ if }y \in E^c \medskip \\
-1 \qquad & \mbox{ if }y \in E.
\end{array}
\right.
$$
If $\mm$ is smooth, then the fractional mean curvature is well defined at each point and is a regular function. In fact, the following
result is known, see \cite[Proposition 6.3]{FFMMM} and \cite[Proposition 2.1]{CFSW}.
\begin{thm}\label{dH}
Suppose $\partial E$ is of class $C^{1,\beta}$, with $\beta>s$. Then the right-hand side of \eqref{defspace} is well defined and finite for all $x \in \partial E$ and defines a continuous function on $\partial E$.  If in addition $\partial E$ is of class $C^{2,\beta}$, with $\beta>s$, then $H_s \in C^1(\partial E)$ and its derivative in a tangential direction $v \in T_x \mm$ is given by
\begin{equation}\label{tang1}
\frac{\partial H_s}{\partial v}(x) = s(1-s)(n+s) \lim_{\ep \to 0^+}\int_{\rl^n \setminus B_\ep(x)} \tilde \chi_E(y) \frac{ \langle y-x, v \rangle}{|x-y|^{n+s+2}} dy.
\end{equation}
\end{thm}

By using the divergence theorem and estimating the boundary terms on $\partial B_\ep(x)$ with techniques similar to the proof of \cite[Proposition 2.1]{CFSW}, we can prove that, under the hypotheses of the previous theorem, $H_s$ and its gradient can be written as boundary integrals on $\mm$, as follows:
\begin{equation}\label{defsurface}
H_s(x) = 2(1-s) \lim_{\ep \to 0^+}\int_{\mm \setminus B_\ep(x)} \frac{\langle y-x,\nu (y) \rangle}{|x-y|^{n+s}} d\mu(y),
\end{equation}
\begin{equation}\label{tang2}
\frac{\partial H_s}{\partial v}(x) = 2s(1-s) \lim_{\ep \to 0^+}\int_{\mm \setminus B_\ep(x)} \frac{\nu(y) \cdot v}{|x-y|^{n+s}} d\mu(y).
\end{equation}


We also recall that the {\em fractional perimeter} of $E$, as introduced in \cite{CRSavin}, is defined as
$$
{\rm Per}_s(E)= s(1-s) \int_{E} \int_{E^c} \frac{dx \, dy}{|x-y|^{n+s}}.
$$
Then fractional mean curvature arises as the first variation of the fractional perimeter along a deformation of $E$, see \eqref{eqper} later.

We state a general criterion for the convergence of singular integrals on the boundary of a smooth compact set $E$. Suppose that $\partial E$ is of class $C^{1,\beta}$, for some $\beta >s$, and that $f \in C^2(\partial E)$. Then, for any given $x \in \partial E$, the quantity
$$
\lim_{\ep \to 0^+}\int_{\mm \setminus B_\ep(x)} \frac{f(y)-f(x)}{|x-y|^{n+s}} d\mu(y)
$$
exists and is finite.This can be proved by standard arguments. Roughly speaking, the contribution of the first order approximation of $f(y)-f(x)$ around $x$ cancels by symmetry reasons. The remaining terms are of order $O(|y-x|)^{1+\beta}$, by the smoothness of $\partial E$ and of $f$, and this ensures convergence of the integral. In the following, for simplicity of notation, we will write singular integrals as the ones above as if they were ordinary integrals, with the implicit meaning that they are taken in the principal value sense. 

We now recall some notation and general results about geometric evolutions of sets and hypersurfaces. Let us  consider a time-dependent family of sets $E_t$ evolving smoothly from a given initial set $E_0$. We can consider the corresponding evolution of the boundaries, and study the map $F:\mm_0 \times [0,T) \to \rl^n$, where $\mm_0=\partial E_0$ and $\mm_t:=\partial E_t$. Let us denote by $V(p,t):=\langle \partial_t F(p,t), \nu(p,t) \rangle$ the normal component of the speed of our flow.

We first recall the properties of the evolution of the classical geometric quantities associated to the hypersurfaces $\mm_t$. As in \cite{Hu84}, we denote by $g_{ij}$ the components of the metric tensor in a given coordinate system, by $g^{ij}$ its inverse, by $h_{ij}$ the second fundamental form, by $H=h_{ij}g^{ij}$ the mean curvature and by $|A|^2=h_{ij}g^{jl}h_{lk}g^{ki}$ the squared norm of the second fundamental form. If $\ld_1 \leq \dots \leq \ld_{n-1}$ denote the principal curvatures at a given point, then $H=\ld_1+\dots+\ld_{n-1}$, while $|A|^2=\ld_1^2+\dots+\ld_{n-1}^2$. We also denote by  $\nabla^{\mm_t}$, $\Delta^{\mm_t}$ respectively the tangential gradient and the Laplace-Beltrami operator defined on $\mm_t$.

We denote by $p,q,\dots$ the points on $\mm_0$ and by $x,y,\cdots$ the points on $\mm_t$ for positive $t$, as well as the general points in $\rl^n$. For simplicity of notation, when considering the speed $V$ on $\mm_t$ for a fixed $t$, we will usually write $V(x)$ with $x \in \mm_t$ instead of $V(p,t)$, with $x=F(p,t)$. We will use similar conventions for all other quantities defined on the evolving hypersurfaces.
We also denote by $d\mu$ the surface measure along ${\mm_t}$.
In this notation, we recall \cite[Theorem 3.2 and
Lemmata 7.4, 7.5 and 7.6]{Huisken-Polden} and we have:

\begin{lem}\label{evol-cl}
The geometric quantities associated to $\mm_t$ satisfy the following equations: \\
{\bf (i)} $\partial_t g_{ij} = 2 V h_{ij} $, \quad $\partial_t g^{ij} = -2 V h^{ij}, $ \\
{\bf (ii)} $\partial_t \, d\mu= V H d\mu, $ \\
{\bf (iii)} $\partial_t \, \nu= - \nabla^{\mm_t} V, $ \\
{\bf (iv)} $\partial_t h_{ij} = - \nabla^{\mm_t}_i \nabla^{\mm_t}_j V + h_{ik}g^{km}h_{mj} V,$ \\
{\bf (v)} $\partial_t H =  -\Delta^{\mm_t} V-|A|^2V$, \medskip \\
{\bf  (vi)} $\dfrac{d}{dt}|E_t| = \displaystyle \int_{\mm_t} V(x) d\mu,$ \quad $\dfrac{d}{dt} |M_t| = \displaystyle \int_{\mm_t} V(x) H(x) d\mu.$
\end{lem}
 
Next we recall the evolution of some nonlocal quantities, see \cite{CRSavin},
\cite[Appendix B, Proposition B.2]{DDPW} and
\cite[Theorem 14]{SaezV}.

\begin{lem}\label{evol-noncl}
{\bf (i)} The fractional perimeter evolves according to
\begin{equation}\label{eqper}
\dfrac{d}{dt} {\rm Per}_s(E_t) = \int_{\mm_t} H_s(x) V(x) d\mu. 
\end{equation}
{\bf (ii)} The fractional mean curvature satisfies the equation
\begin{equation}\label{eqH}
\frac{\partial_t H_s}{2s(1-s)}
 =  - \int_{\mm_t} \frac{V(y)-V(x)}{|y-x|^{n+s}} d\mu(y) - V(x) 
\int_{\mm_t} \frac{1- \nu(y) \cdot \nu(x)}{|y-x|^{n+s}} d\mu(y).
\end{equation}
\end{lem}
 
 We remark that there is a clear analogy between these equations and their classical counterparts. Indeed, as proved in 
\cite[Appendix A]{DDPW}, we have, for a general smooth function $f$ defined on a (fixed) hypersurface $\mm$,
  \begin{equation}\label{asymptlaplace}
 \lim_{s \to 1^-} 2s(1-s) \int_{\mm} \frac{f(y)-f(x)}{|y-x|^{n+s}} d\mu(y) = \omega_n \Delta^\mm f(x),
  \end{equation}
 where $\omega_n$ is the volume of the unit ball of $\rl^n$. In addition,
 \begin{equation}\label{asymptasquare}
 \lim_{s \to 1^-} 2s(1-s) \int_{\mm} \frac{1- \nu(y) \cdot \nu(x)}{|y-x|^{n+s}}  d\mu(y) = \omega_n |A|^2.
 \end{equation}
 
From now on, we assume that the map $F:\mm_0 \times [0,T) \to \rl^n$ satisfies equation \eqref{flow}. This corresponds to the normal speed
$$
V(p,t)=-H_s(p,t)+h(t),
$$
with $h(t)$ defined as in \eqref{volpr}. 

Then Lemma \ref{evol-cl}-(vi) implies that the enclosed volume $E_t$ remains constant in time, while by Lemma \ref{evol-noncl}-(i) the fractional perimeter decreases according to
\begin{equation}\label{66}
\partial_t {\rm Per}_s(E_t) = \int_{\mm_t}  [- H_s(x) +h(t)] H_s(x) d\mu = - \int_{\mm_t} [H_s(x) - h(t)]^2 d\mu \leq 0.
\end{equation}

We conclude this section recalling the analogue of Alexandrov Theorem in the nonlocal setting.

\begin{thm}[Theorem 1.1 in \cite{CFSW}, Theorem 1.1 in \cite{Ciraolo-etal}]\label{Alex}
Let $E$ be a bounded open set of class $C^{1,s}$ and with constant nonlocal mean curvature. Then, $E$ is a ball.
 \end{thm}
 
We point out that, by~\eqref{66} and Theorem~\ref{Alex},
the monotonicity of~${\rm Per}_s(E_t)$ is strict unless~$E_t$
is a sphere.

\section{Bounds on inner and outer radii}\label{SE:3}

Given a bounded set~$E\subset\R^n$ with nonempty
interior and~$\omega\in \partial B_1$, we denote by~$
w_E(\omega)$ the width of the set~$E$ in direction~$\omega$, i.e.
\begin{equation}\label{DEF:w}
w_E(\omega):=\sup_{x,y\in E} (x-y)\cdot\omega.\end{equation}
Notice that~$w_E$ is
the distance between the two hyperplanes orthogonal to~$\omega$
touching~$E$ from outside. 
We also set
$$ \underline{w_E}:=\inf_{\omega\in \partial B_1} w_E(\omega)
\qquad{\mbox{ and }}\qquad 
\overline{w_E} :=\sup_{\omega\in \partial B_1} w_E(\omega).$$
By construction, we have
\begin{equation}\label{diam:eq}
\overline{w_E} \,=\,{\rm diam}\,(E).\end{equation}
Recalling the notation in~\eqref{RADIO},
if~$E$ is convex, it is known that
\begin{equation}\label{5.4} \underline{\rho_E}\ge\frac{ \underline{w_E} }{n+1}
\qquad{\mbox{ and }}\qquad
\overline{\rho_E}\le\frac{ \overline{w_E} }{\sqrt{2}},\end{equation}
see e.g. Lemma~5.4 in~\cite{Andrews-CalcVar-94}.

Using this notation, the following result holds true:

\begin{prop} \label{bounds}
For any bounded, convex set~$E\subset\R^n$
with nonempty interior, we have that
\begin{eqnarray}
\label{DES:PR:1}
&& \underline{w_E}\ge c\,\left( \frac{|E|}{{\rm Per}_s(E)}\right)^{\frac{1}{s}},
\\
\label{DES:PR:2}
&& \underline{\rho_E}\ge c\,\left(
\frac{|E|}{{\rm Per}_s(E)}\right)^{\frac{1}{s}},
\\
\label{DES:PR:3}
&& \overline{w_E}\le C\,\left( {\rm Per}_s(E)\right)^{\frac{n-1}{s}}
\, |E|^{\frac{1+s-n}{s}},\\
\label{DES:PR:4}&&
\overline{\rho_E}\le C\,\left( {\rm Per}_s(E)\right)^{\frac{n-1}{s}}
\, |E|^{\frac{1+s-n}{s}}\\
\label{DES:PR:5}
{\mbox{and }}&&
\frac{\;\;\overline{\rho_E}\;\;}{\;\;\underline{\rho_E}\;\;}\le C\,
\left( {\rm Per}_s(E)\right)^{\frac{n}{s}}
\,|E|^{\frac{s-n}{s}}
\end{eqnarray}
for suitable constants~$C>c>0$ only depending on $n,s$. 
\end{prop}

\begin{proof} 
First of all, we observe that
\begin{equation}\label{ENOUGH}
{\mbox{it is enough to prove~\eqref{DES:PR:1},}}
\end{equation}
since, after that, the claims in~\eqref{DES:PR:2}, \eqref{DES:PR:3},
\eqref{DES:PR:4} and~\eqref{DES:PR:5}
would follow.
Indeed, if~\eqref{DES:PR:1} holds true, then~\eqref{DES:PR:2}
follows directly from~\eqref{5.4}.

Now we prove~\eqref{DES:PR:3} assuming that~\eqref{DES:PR:1}
(and so~\eqref{DES:PR:2})
holds true. To this aim, we 
observe that we can suppose that
\begin{equation} \label{78:A0}
\overline{w_E}\ge 4\underline{\rho_E}.
\end{equation}
Indeed, suppose instead that the opposite inequality holds. Then, we
use the nonlocal isoperimetric inequality (see~\cite{MR2425175})
to see that
$$ |E|^\frac{1}{n}=|E|^\frac{(n-s)(n-1)}{ns}\, |E|^\frac{1+s-n}{s}
\le C_1\,\left( {\rm Per}_s(E)\right)^\frac{n-1}{s}\,|E|^\frac{1+s-n}{s}
,$$											
for some $C_1>0$. Accordingly, since $|E|^\frac{1}{n}
\ge |B_{\underline{\rho_E}}|^\frac{1}{n}=
C_2\,{\underline{\rho_E}\,}$, for some $C_2>0$, we obtain
$$ C_2\,{\underline{\rho_E}\,}\le
C_1\,\left( {\rm Per}_s(E)\right)^\frac{n-1}{s}\,|E|^\frac{1+s-n}{s}$$
and so, if the opposite inequality
holds in \eqref{78:A0},
$$ \frac{C_2\,{\overline{w_E}\,}}{4}\le
C_1\,\left( {\rm Per}_s(E)\right)^\frac{n-1}{s}\,|E|^\frac{1+s-n}{s},$$
which says that~\eqref{DES:PR:3} is satisfied.

Consequently, we may assume that
\eqref{78:A0} holds true.
Thus, after a translation
we may suppose that~$B_{\underline{\rho_E}}\subseteq E$
and there exists~$p\in\overline{E}$
with~$|p|\ge \frac{\overline{w_E}}{2}-\underline{\rho_E}$.
We stress that, in view of~\eqref{78:A0},
$$ |p|\ge \frac{\overline{w_E}}{4}=:\ell.$$
Since~$E$ is convex, the convex hull of~$p$ with~$B_{\underline{\rho_E}}$
lies in~$\overline{E}$ and therefore~$|E|\ge \tilde{c} \,
{\underline{\rho_E}\,}^{n-1}\ell$, for some~$\tilde{c}>0$.
This and~\eqref{DES:PR:2} imply that
$$ \overline{w_E}=4
\ell \le
\frac{4\,|E|}{\tilde{c} \,{\underline{\rho_E}\,}^{n-1}}
\le
\frac{4\,|E|\,\left( {\rm Per}_s(E)\right)^{\frac{n-1}s}}{\tilde{c} \,c^{n-1}\,
|E|^{\frac{n-1}{s}}},$$
which gives~\eqref{DES:PR:3}, as desired.

Then, from~\eqref{DES:PR:3}
and~\eqref{5.4}, one obtains~\eqref{DES:PR:4}.
Finally, \eqref{DES:PR:5}
clearly follows from~\eqref{DES:PR:2}
and~\eqref{DES:PR:4}. This completes the proof of~\eqref{ENOUGH}.

In view of~\eqref{ENOUGH}, from now on
we focus on the proof of~\eqref{DES:PR:1}.
To this aim, after a rigid motion, we may suppose that
$\underline{w_E}$ is realized in the vertical direction, and, more
precisely, that
\begin{equation} \label{78:A00}
E\subseteq \big\{x_n \in [-\underline{w_E}, \,0]\big\}.
\end{equation}
We denote by $\pi$ the projection onto $\R^{n-1}\times\{0\}$
and $E':=\pi(E)$. We consider a nonoverlapping tiling of $\R^{n-1}\times\{0\}$
by cubes $\{ Q_i\}_{i\in\N}$ which have side equal to ${\underline{w_E}}
\,/{\sqrt{n-1}}$
(hence, their diagonal is equal to $\underline{w_E}$).
We denote by $\N_\star$ the set of indices $i\in\N$ for which
$Q_i$ intersects $E'$. Let also
$$ Q:=\bigcup_{i\in\N_\star} Q_i
\qquad{\mbox{ and }}\qquad
F:= Q\times \big(0,\,\underline{w_E}\big].$$
Due to \eqref{78:A00}, we know that $F$ lies outside $E$ and therefore
\begin{eqnarray*}
&& {\rm Per}_s(E)\\
&\ge& \iint_{E\times F} \frac{dx\,dy}{|x-y|^{n+s}}\\
&=& \int_{-\underline{w_E}}^{0} dx_n\;
\int_{Q} dx'\;
\int_0^{\underline{w_E}} dy_n\;
\int_{Q} dy'\;\frac{\chi_E(x',x_n)}{|x-y|^{n+s}}\\
&\ge& \sum_{i\in\N_\star}\int_{-\underline{w_E}}^{0} dx_n\;
\int_{Q_i} dx'\;
\int_0^{\underline{w_E}} dy_n\;
\int_{Q_i} dy'\;\frac{\chi_E(x',x_n)}{|x-y|^{n+s}}.
\end{eqnarray*}
Now we remark that if 
$x'$, $y'\in Q$,
$x_n\in \big[-\underline{w_E},\,0\big]$
and $y_n\in\big(0,\,\underline{w_E}\big]$, we have that
$$ |x-y|^2=|x'-y'|^2+|x_n-y_n|^2
\le {\underline{w_E}\,}^2+(2{\underline{w_E}\,})^2
= 5\,{\underline{w_E}\,}^2.$$
As a consequence,
\begin{eqnarray*}
&& {\rm Per}_s(E)\\
&\ge& \frac{1}{
5^{\frac{n+s}{2}}\;
{\underline{w_E}\,}^{n+s} }\,
\sum_{i\in\N_\star}\int_{-\underline{w_E}}^{0} dx_n\;
\int_{Q_i} dx'\;
\int_0^{\underline{w_E}} dy_n\;
\int_{Q_i} dy'\; {\chi_E(x',x_n)}
\\
&=& \frac{1}{5^{\frac{n+s}{2}}\;
{\underline{w_E}\,}^{n+s} }\,
\left( \frac{\quad\underline{w_E}\quad}{\sqrt{n-1}}\right)^{n-1}\,\underline{w_E}\;
\sum_{i\in\N_\star}\int_{-\underline{w_E}}^{0} dx_n\;
\int_{Q_i} dx'\;
{\chi_E(x',x_n)} 
\\
&=& \frac{1}{5^{\frac{n+s}{2}}\;
{\underline{w_E}\,}^{n+s} }\,
\left( \frac{\quad\underline{w_E}\quad}{\sqrt{n-1}}\right)^{n-1}\;\underline{w_E}\,
|E|,
\end{eqnarray*}
where we used \eqref{78:A00} once again in the last identity.
This estimate plainly implies \eqref{DES:PR:1}, as desired.
\end{proof}

For completeness, we point out an interesting geometric consequence
of the estimate in \eqref{DES:PR:5} in terms of
the nonlocal isoperimetric ratio
$$ {\mathcal{I}}_s(E):=
\frac{ \left( {\rm Per}_s(E)\right)^{n} }{ |E|^{n-s}}.$$
Indeed, formula \eqref{DES:PR:5} states that
if the nonlocal isoperimetric ratio of $E$ is bounded, then so
is the ratio between the inner and outer radius of $E$ and,
more precisely
$$ \frac{\;\;\overline{\rho_E}\;\;}{\;\;\underline{\rho_E}\;\;}\le C\,
\left( {\mathcal{I}}_s(E) \right)^{\frac1s}.$$
In the local case when $s=1$, this formula
was already
known, see e.g. Proposition 5.1 \cite{Andrews-aniso} or
Proposition 2.1 in \cite{Sinestrari-CalcVar-2015}.

As an immediate consequence of the results of this section, we obtain

\begin{cor}\label{spittelmarkt}

Let $E_0$ be a convex subset of $\R^n$ and $\mm_0=\partial E_0$. Let $F:\mm_0 \times [0,T) \to \rl^n$, with $0 <T \leq +\infty$, be a  solution of \eqref{flow}. Then there exist positive constants $0 <R_1 \leq R_2$, only depending on $E_0$, such that
$$
R_1 \leq \underline{\rho_{E_t}} \leq \overline{\rho_{E_t}} \leq R_2, \qquad \forall t \in [0,T).
$$
In addition, there exists $K_1>0$ such that $H_s(p,t) \geq K_1$ for all $(p,t) \in \mm_0 \times [0,T)$.
\end{cor}
\begin{proof}
As already mentioned in the Introduction, we know that the evolution given by \eqref{flow} preserves convexity, as established in \cite{CNRuffini}, hence we have that $E_t$ is convex for all $0<t<T$.

By definition, we have
$$
\omega_n \underline{\rho_{E_t}}^n \leq |E_t| \leq \omega_n \overline{\rho_{E_t}}^n.
$$
Since $|E_t|$ is constant, this gives an upper bound on $\underline{\rho_{E_t}}$ and a lower bound on $\overline{\rho_{E_t}}$ in terms of $|E_0|$. On the other hand, since ${\rm Per}_s(E_t)$ is decreasing in time, inequality \eqref{DES:PR:5} gives a uniform bound on the ratio 
$\overline{\rho_{E_t}} /\underline{\rho_{E_t}}$. These properties together yield the first assertion.

To prove the lower bound on $H_s$, let us consider an arbitrary point~$x \in \mm_t$. Since $E_t$ is convex, it is contained in the half-space $\{ y \in \rl^n \, : \, (y-x) \cdot \nu(x) \leq 0 \}$. Moreover, by definition, the diameter of $E_t$ is not greater than $2 \overline{\rho_{E_t}}$ which is less than $2R_2$. Therefore, if we introduce the half-balls
$$
B_+=\{ y \in B_{2R_2}(x) \, : \, (y-x) \cdot \nu(x) \geq 0 \}, \qquad 
B_-=\{ y \in B_{2R_2}(x) \, : \, (y-x) \cdot \nu(x) \leq 0 \}, 
$$
we have that $E_t \subset B_-$. It follows that
\begin{eqnarray*}
\frac{1}{s(1-s)}H_s(x) & = & \int_{E_t^c} \frac{dy}{|x-y|^{n+s}}  -  \int_{E_t} \frac{dy}{|x-y|^{n+s}}  \\
& \geq & \int_{\rl^n \setminus B_{2R_2}(x)} \frac{dy}{|x-y|^{n+s}}  + \int_{B_+} \frac{dy}{|x-y|^{n+s}}  - \int_{B_-} \frac{dy}{|x-y|^{n+s}}  \\
& = & \int_{\rl^n \setminus B_{2R_2}(x)} \frac{dy}{|x-y|^{n+s}}  = \int_{|z| \geq 2R_2} \frac{dz}{|z|^{n+s}},
\end{eqnarray*}
where the last integral is independent on $x,t$.
\end{proof}

The previous result contains the first part of the statement of Theorem \ref{main} (the bounds on inner and outer radii and the lower bound for $H_s$). To conclude the proof of
Theorem~\ref{main} it remains to establish the upper bound for the fractional mean curvature, which will be done in
Section~\ref{S5}.

We conclude this section with the following observation. Corollary \ref{spittelmarkt} ensures that, at any given time, there exists a ball of radius $R_1$ contained in $E_t$. However, the center of the ball may be different at different times. We want to show that, by choosing a smaller radius, we can find a ball with fixed center which remains inside $E_t$ for a time interval with fixed length.

\begin{lem}\label{sphinside}
For any $t_0 \geq 0$, we can find $x_0 \in \rl^n$ such that
$$
B_{\frac{R_1}{2}} (x_0) \subset E_t, \qquad \forall t \in [t_0,t_0+t^*]
$$ 
where $t^*>0$ only depends on $n,s,R_1$.
\end{lem}
\begin{proof}
As in \cite{Andrews-aniso,McCoy2004}, we use a comparison argument. Volume preserving curvature flows in general do not satisfy an avoidance principle. However, if $E_t$ evolves by \eqref{flow} and $F_t$ evolves by the standard fractional mean curvature flow (corresponding to $h(t) \equiv 0$) then an easy maximum principle argument shows that if $F_{t_0} \subset E_{t_0}$ at a certain time $t_0$, then we also have $F_t \subset E_t$
for all $t \geq t_0$. 

In our case, we can use comparison with a shrinking ball. From the previous corollary, there exists $x_0$ such that $B_{R_1} (x_0) \subset E_t$. We set $F_{t_0} = B_{R_1} (x_0)$ and we denote by $F_t$ the evolution of  $F_{t_0}$ for $t \geq t_0$ by standard fractional mean curvature flow, which is a shrinking sphere. We let $t^*$ the time such that  $F_{t_0+t^*} = B_{R_1/2} (x_0)$, whose value only depends on $n,s,R_1$. Then the comparison argument yields the conclusion. 
\end{proof}

\section{Integral surface estimates for convex sets}\label{S4}

We collect in this section some estimates
on weighted integrals along the boundary of a convex set.
We start with a uniform estimate of the weighted surface of
a convex set only in dependence of its
inner and outer radii.

\begin{lem}\label{BETA}
Let $\beta>1$ and let $E\subset\R^n$ be a bounded, convex set with nonempty interior. Then, there exists a constant~$C>0$,
depending on~$n$, such that, for any $\in\partial E$, we have
$$ \int_{\partial E}
\frac{d\mu(y)}{|x-y|^{n-\beta}} \le C\,
\frac{\;\;\overline{\rho_E}\;\;}{\;\;\underline{\rho_E}\;\;}
\,\left[
\frac{1}{\beta-1} + \left(\frac{\;\;\overline{\rho_E}\;\;}{\;\;\underline{\rho_E}\;\;}\right)^{n-2}\, \right]
\big( {{\rm diam}\,(E)}\big)^{\beta-1}
.$$
\end{lem}

\begin{proof} 
We can suppose that $x$ is the origin. 
By definition, there exists~$p\in E$ such that~$
B_{ \underline{\rho_E} }(p)\subseteq E$. By convexity,
the convex envelope of~$0$ and~$B_{ \underline{\rho_E} }(p)$
lies in~$\overline{E}$. Up to a rotation, we can assume that
$p=(0,\dots,|p|)$. This easily implies, again by convexity, that 
$B_{ \underline{\rho_E}/2}(0)\cap \partial E$ is the graph
of a Lipschitz function~$f$, with Lipschitz constant bounded by~$2|p|/  \underline{\rho_E} \leq 4\overline{\rho_E} /\underline{\rho_E}$.

Let us set~$\delta:= \underline{\rho_E}/2$
and~$M := \overline{\rho_E} /\underline{\rho_E}$. In addition, let us denote by $C',C'',\dots$ constants depending only on $n$. We can estimate, using the fact that~$\beta>1$,


\begin{equation}\label{09ojATYUA}
\begin{split}
& \int_{\partial E\cap B_{\delta}}
\frac{d\mu(y)}{|y|^{n-\beta}}
\le \int_{ {y'\in\R^{n-1}}\atop{|y'|\le \delta } }
\frac{\sqrt{1+|\nabla f(y')|^2}}{|y'|^{n-\beta}}\,dy'
\\ &\qquad
\le C' M\, \int_0^{\delta} \frac{\tau^{n-2}}{
\tau^{n-\beta}}\,d\tau
= \frac{C' M}{\beta-1}\,\delta^{\beta-1}.
\end{split}
\end{equation}
The remaining part of the integral satisfies
\begin{equation}\label{rem}
 \int_{\partial E\setminus B_\delta}
\frac{d\mu(y)}{|y|^{n-\beta}}\le
 \frac{1}{\delta^{n-\beta}} 
\int_{\partial E\setminus B_\delta}
d\mu(y)
\le 
 \frac{\mu(
\partial E )}{\delta^{n-\beta}}.
\end{equation}
Now we observe that
\begin{equation}\label{FUORI}
\mu(\partial E)\le
\mu(B_{ \overline{\rho_E} } ).
\end{equation}
Indeed, we know that there exists~$q\in E$
such that~$B_{ \overline{\rho_E} }(q)\supseteq E$.
Let us denote by $\Pi_E : \rl^n \to E$ the projection on the convex set $E$.
Then $\Pi_E$ maps  $\partial B_{ \overline{\rho_E} }(q)$
onto $\partial E$ and is nonexpansive,  from which~\eqref{FUORI}
follows.

As a consequence of~\eqref{rem} and \eqref{FUORI}, we obtain that
$$  \int_{\partial E\setminus B_\delta}
\frac{d\mu(y)}{|y|^{n-\beta}}
\leq \frac{C''  {\, \overline{\rho_E} \, }^{n-1}}{\delta^{n-\beta}} 
= C'''  M^{n-1} \delta^{\beta-1}.$$
This and~\eqref{09ojATYUA} imply the desired result
(recall also~\eqref{diam:eq} and~\eqref{5.4}).
\end{proof}

Now we obtain a bound on the fractional mean curvature 
in terms of the integral quantity which appears in the last term of \eqref{eqH}.
In view of \eqref{asymptasquare}, one can consider this
estimate as the fractional counterpart of the elementary property that
the classical mean curvature is bounded by the norm of the second
fundamental form. An estimate of this kind is more delicate to obtain
in the nonlocal case, since the
fractional mean curvature
cannot be realized by the average of finitely many
directional curvatures, and so methods involving linear
algebra cannot be applied, see~\cite{Abatangelo}.
We give here a proof in the case of convex sets, but it is natural
to expect that a similar property should hold in a more general setting. 

\begin{prop}\label{hausvogteiplatz}
Let~$E\subset\R^n$ be a convex set with~$C^{1,\alpha}$ boundary,
with~$\alpha\in(s,1)$. Then, there exists~$C>0$, depending on~$n$ and on the ratio $\overline{\rho_E}/
\underline{\rho_E}$, such that, for every~$x\in\partial E$, we have
$$ H_s(x) \le C\, \big( {{\rm diam}\,(E)}\big)^{\frac{1-s}2}
\left( (1-s)
\int_{\partial E} \frac{1-\nu(y)\cdot\nu(x)}{|x-y|^{n+s}}\,d\mu(y) \right)^{\frac 12}.$$
\end{prop}

\begin{proof} Given~$x\in\partial E$, with exterior normal~$\nu(x)$,
from the convexity of~$E$
we have that~$\{p\in\R^n{\mbox { s.t. }}(p-x)\cdot\nu(x)>0\}$
touches~$E$ from outside at~$p$.
As a consequence, if~$y\in\partial E$, we have that~$(y-x)\cdot\nu(x)\le0$
and therefore, recalling \eqref{defsurface}, we have
\begin{eqnarray*}
\frac{1}{2(1-s)}H_s(x) &=& \int_{\partial E} \frac{(y-x)\cdot\nu(y)}{
|x-y|^{n+s}}\,d\mu(y)\\
&=& \int_{\partial E} \frac{(y-x)\cdot\nu(x)}{
|x-y|^{n+s}}\,d\mu(y)+
\int_{\partial E} \frac{(y-x)\cdot\big(\nu(y)-\nu(x)\big)}{
|x-y|^{n+s}}\,d\mu(y)
\\ &\le&
\int_{\partial E} \frac{(y-x)\cdot\big(\nu(y)-\nu(x)\big)}{
|x-y|^{n+s}}\,d\mu(y)
\\ &\le&
\int_{\partial E} \frac{\big|\nu(y)-\nu(x)\big|}{
|x-y|^{n+s-1}}\,d\mu(y)
\\ &=&
\int_{\partial E} \frac{\big|\nu(y)-\nu(x)\big|}{
|x-y|^{\frac{n+s}{2}}}\;\frac{d\mu(y)
}{
|x-y|^{\frac{n+s-2}{2}}
}
.\end{eqnarray*}
Hence, exploiting the
H\"older's Inequality,
$$ \frac{1}{2(1-s)}H_s(x)  \le
\sqrt{
\int_{\partial E} \frac{\big|\nu(y)-\nu(x)\big|^2}{
|x-y|^{{n+s}}} \,d\mu(y)}\quad
\sqrt{
\int_{\partial E} 
\frac{d\mu(y)
}{
|x-y|^{{n+s-2}}
} }.$$
Since we have
$
\big|\nu(y)-\nu(x)\big|^2 = 2(1- \nu(y) \cdot \nu(x)),
$ the desired result follows easily from Lemma~\ref{BETA} with~$\beta:=2-s>1$.
\end{proof}

\section{Upper bound on the fractional curvature}\label{S5}

In this section, we show that the bounds on the inner and outer radii imply that the fractional mean curvature of our solution is bounded from above. This, together with Corollary \ref{spittelmarkt}, will conclude the proof of Theorem \ref{main}.

To this purpose, we adapt to the nonlocal setting a technique originally introduced in \cite{Tso}. We consider the support function on the evolving hypersurface
$$
u(p,t)= \langle \, F(p,t) \, , \, \nu(p,t) \, \rangle.
$$
 
By Lemma \ref{evol-cl}-(iii) and the representation \eqref{tang2} of the gradient of $H_s$, we find that $u$ evolves according to 
\begin{eqnarray}
\partial_t u &  = & \langle \, \partial_t F \, , \, \nu \, \rangle + \langle \, F \, , \, \partial_t \nu \rangle \nonumber \\
& = & -H_s + h + \langle \, F \, , \, \nabla^\mm H_s \rangle \nonumber \\
& = &-H_s + h + 2s(1-s)\int_{\mm_t} \frac{x^T \cdot \nu(y)}{|y-x|^{n+s}} d\mu(y). \label{evfspt}
\end{eqnarray}

{F}rom Lemma \ref{sphinside}, we know that for any $t_0$ there exists $x_0 \in \rl^n$ such that $B_{R_1/2}(x_0) \subset E_t$ for any $t \in [t_0,t_0+t^*]$. For simplicity, we perform our computations in the case $x_0=0$. By the convexity of $E_t$, we deduce that $u \geq R_1/2$ on $\mm_t$ for all $t \in [t_0,t_0+t^*]$. We then set $\alpha=R_1/4$ and we consider the function
$$
W = \frac{H_s}{u-\alpha}.
$$


Since
$$
\alpha \leq u-\alpha \leq {\rm diam}(\mm_t) - \alpha \leq 2\overline{\rho_{E_t}} -\alpha,
$$
we deduce from Corollary \ref{spittelmarkt} that
\begin{equation}
\label{estWH}
\frac 1C \leq \frac{W}{H_s} \leq C,
\end{equation}
for some $C$ only depending on the $n,s$ and the initial data.

Let us now analyze the evolution equation satisfied by $W$. By Lemma \ref{evol-noncl}-(ii), 
the fractional mean curvature satisfies the equation
$$
\frac{\partial_t H_s}{2s(1-s)}
 =  \int_{\mm_t} \frac{H_s(y)-H_s(x)}{|y-x|^{n+s}} d\mu(y) +(H_s(x)-h(t)) 
\int_{\mm_t} \frac{1- \nu(y) \cdot \nu(x)}{|y-x|^{n+s}} d\mu(y).
$$
Recalling \eqref{evfspt} and neglecting the positive terms containing $h(t)$, we find
\begin{eqnarray}
\frac{\partial_t W(x,t)}{2s(1-s)} & = & \frac{1}{u(x)-\alpha} \int_{\partial E_t} \frac{H_s(y)-H_s(x)}{|y-x|^{n+s}} d\mu(y) + \frac{H_s(x)-h(t)}{u(x)-\alpha}
\int_{\partial E_t} \frac{1- \nu(y) \cdot \nu(x)}{|y-x|^{n+s}} d\mu(y) \nonumber \\
& & -\frac{H_s(x)}{(u(x)-\alpha)^2} \left( \frac{-H_s(x)+h(t)}{2s(1-s)} + \int_{\partial E_t} \frac{x^T \cdot \nu(y)}{|y-x|^{n+s}} d\mu(y)\right) \nonumber \\
& < & \frac{1}{u(x)-\alpha} \int_{\partial E_t} \frac{H_s(y)-H_s(x)}{|y-x|^{n+s}} d\mu(y) + \frac{H_s(x)}{u(x)-\alpha}
\int_{\partial E_t} \frac{1- \nu(y) \cdot \nu(x)}{|y-x|^{n+s}} d\mu(y) \nonumber \\
& & -\frac{H_s(x)}{(u(x)-\alpha)^2} \left( \frac{-H_s(x)}{2s(1-s)} + \int_{\partial E_t} \frac{x^T \cdot \nu(y)}{|y-x|^{n+s}} d\mu(y) \right). \label{eprimo}
\end{eqnarray}

We can rewrite 
\begin{equation}
\int_{\partial E_t} \frac{H_s(y)-H_s(x)}{|y-x|^{n+s}} d\mu(y)
 = \int_{\partial E_t} (u(y)-\alpha) \frac{W(y)-W(x)}{|y-x|^{n+s}} d\mu(y) +
W(x) \int_{\partial E_t} \frac{u(y)-u(x)}{|y-x|^{n+s}} d\mu(y). \label{esecondo}
\end{equation}

Observe also
\begin{eqnarray}
\lefteqn{\int_{\partial E_t} \frac{u(y)-u(x)}{|y-x|^{n+s}} d\mu(y)
 =  
\int_{\partial E_t} \frac{y \cdot \nu(y) - x \cdot \nu(x)}{|y-x|^{n+s}} d\mu(y)} \nonumber \\
& = & \int_{\partial E_t} \frac{(y - x) \cdot \nu(y)}{|y-x|^{n+s}} d\mu(y) +
\int_{\partial E_t} \frac{ ( x^T + u(x) \nu(x) ) \, \cdot \, ( \nu(y) -\nu(x) )}{|y-x|^{n+s}} d\mu(y) \nonumber \\
& = & \frac{1}{2(1-s)} H_s(x) +
\int_{\partial E_t} \frac{ x^T \cdot  \nu(y) }{|y-x|^{n+s}} d\mu(y)
- u(x) \int_{\partial E_t} \frac{ 1- \nu(x) \cdot \nu(y)}{|y-x|^{n+s}} d\mu(y). \label{eterzo}
\end{eqnarray}

{F}rom \eqref{esecondo} and~\eqref{eterzo} we deduce that
\begin{eqnarray*}
&& \frac{1}{u(x)-\alpha} \int_{\partial E_t} \frac{H_s(y)-H_s(x)}{|y-x|^{n+s}} d\mu(y) -\frac{H_s(x)}{(u(x)-\alpha)^2} \int_{\partial E_t} \frac{x^T \cdot \nu(y)}{|y-x|^{n+s}} d\mu(y) \\
& = & 
\frac{1}{u(x)-\alpha}\int_{\partial E_t} \frac{(W(y)-W(x))(u(y)-\alpha)}{|y-x|^{n+s}} d\mu(y) +  \frac{1}{2(1-s)}  W^2 \\
&& -u(x)\frac{W(x)}{u(x)-\alpha} \int_{\partial E_t} \frac{ 1 - \nu(x) \cdot \nu(y)}{|y-x|^{n+s}} d\mu(y).
\end{eqnarray*}


We then conclude from \eqref{eprimo}
\begin{eqnarray}\label{W_t}
\frac{\partial_t W}{2s(1-s)}
& < &  \frac{1}{u(x)-\alpha}\int_{\partial E_t} \frac{(W(y)-W(x))(u(y)-\alpha)}{|y-x|^{n+s}} d\mu(y) \nonumber \\
& &  +  \frac{1+s}{2s(1-s)}  W^2
-\alpha\frac{W(x)}{u(x)-\alpha} \int_{\partial E_t} \frac{ 1 - \nu(x) \cdot \nu(y)}{|y-x|^{n+s}} d\mu(y).
\end{eqnarray}

Recalling the estimate of Proposition \ref{hausvogteiplatz} and \eqref{estWH}, we immediately obtain

\begin{cor}
At any point where the spatial maximum for $W(\cdot,t)$ is attained, we have
\begin{equation} \label{ineqW}
 \partial_t W < C_1 W^2 - C_2 W^3
 \end{equation}
 for constants $C_1,C_2$ only depending on $n,s$ and the initial data.
\end{cor}

We are now ready to prove the upper bound on the fractional mean curvature.

\begin{thm}\label{upper-bound}
Let $E_0$ be a convex subset of $\R^n$ and $\mm_0=\partial E_0$. Let $F:\mm_0 \times [0,T) \to \rl^n$, with $0 <T \leq +\infty$, be a  solution of \eqref{flow} of class $C^{2,\beta}$ for some $\beta>s$. Then there exists $K_2>0$, only depending on $n,\,s,\,E_0$, such that
$$
H_s(p,t) \leq K_2 \qquad p \in \mm_0,
$$
for all $t \in [0,T)$.
\end{thm}

\begin{proof}
Let us take an arbitrary $t_0 \in [0,T)$. We know  from Lemma \ref{sphinside} that there exists $x_0 \in \rl^n$ such that $B_{R_1/2}(x_0) \subset E_t$ for any $t \in [t_0,t_0+t^*]$. In addition, setting $W=H_s(\langle x-x_0,\nu\rangle-R_1/4)^{-1}$, we know that the maximum of $W$ satisfies inequality \eqref{ineqW} in this time interval. We need a little care because the point $x_0$ depends on $t_0$ and therefore the function $W$ is defined differently in different intervals.

Let us set for simplicity $F(w)=C_1w^2-C_2 w^3$ to denote the right-hand side of \eqref{ineqW}.  We observe that $F(w) <0$ for $w > C_1/C_2$. Let us denote by $\tilde w(t)$ the solution of the equation $\tilde w'(t)=F(\tilde w(t))$ defined for $t > 0$ and satisfying $\tilde w(t) \to +\infty$ as $t \to 0^+$. It is easily seen that such a function exists and is implicitly defined by the formula
$$
\int_{\tilde w(t)}^{+\infty} \frac{dw}{C_2w^3-C_1w^2}= t.
$$
In addition, $\tilde w(t)$ is defined for all $t \in (0,+\infty)$ and decreases monotonically from $+\infty$ to $C_1/C_2$. 

We now treat differently the cases $t_0=0$ and $t_0>0$. If $t_0=0$, using the sign properties of the right-hand side of  \eqref{ineqW}, we obtain
$$
W(p,t)  \leq \max \left \{ \max_{\mm_0} W, \frac{C_1}{C_2} \right \}, \qquad p \in \mm, t \in [0,t^*].
$$
Keeping into account \eqref{estWH}, this implies 
\begin{equation}
\label{est1}
H_s(p,t) \leq C',  \qquad p \in \mm, t \in [0,t^*],
\end{equation}
for a suitable constant $C'$. If $t_0>0$, we observe instead that, again by \eqref{ineqW},
$$
W(p,t_0+\tau) \leq \tilde w(\tau), \qquad \tau \in [0,t^*].
$$
In particular, since $\tilde w$ is monotone,
$$
W(p,t_0+\tau) \leq \tilde w(t^*/2), \qquad \tau \in [t^*/2,t^*].
$$
Using \eqref{estWH},
it follows that
\begin{equation}
\label{est2}
H_s(p,t) \leq C'',  \qquad p \in \mm, t \in \left[t_0+\frac{t^*}{2},t_0+t^*\right].
\end{equation}
By the arbitrariness of $t_0$, we conclude from \eqref{est1}--\eqref{est2} that $H_s(p,t) \leq K_2:=\max\{C',C''\}$, for all $p,t$.
\end{proof}

\section{The case of a nonlinear speed} \label{S6}
In this section we study a generalization of problem \eqref{flow}, in which the velocity is given by a general function of the fractional mean curvature. More precisely, we consider
\begin{equation}\label{non-homo}
\begin{cases}
\partial_t F(p,t)=[-\Phi(H_s(p,t))+\varphi(t)]\nu(p,t),& p\in \mathcal M_0, \,\, t\ge 0\\
F(p,0)=p & p \in \mathcal M_0,
\end{cases}
\end{equation}
where
$$\varphi(t)=\frac{1}{|\mathcal M_t|}\int_{\mathcal M_t}\Phi(H_s(x))d\mu.$$

We assume that  $\Phi: [0,+\infty) \rightarrow [0,+\infty)$ is a $C^2$ function, satisfying the following properties:
\begin{itemize}
\item [i)] $\lim_{a\rightarrow + \infty}\Phi(a)=+\infty$,\\
\item [ii)] $\Phi'(a)>0$ for every $a>0$,\\
\item [iii)] $ \lim_{a\rightarrow +\infty}\frac{\Phi'(a)a^2}{\Phi(a)}=+\infty.$\\
\end{itemize}

Typical examples are functions of the form $\Phi(a)=a^p$ with $p>0$,  but hold in many other cases, e.g.  $\Phi(a)=e^a$ or $\Phi(a)=\ln(a+1)$. 
 Assumption (ii) ensures that $\Phi(H_s)$ satisfies the monotonicity assumption (A) in \cite{CMPonsiglione}, Section 2 (monotonicity with respect to set inclusion). Hence, by Theorem~2.21 in \cite{CMPonsiglione}, problem \eqref{non-homo} is well posed and admits a viscosity solution, at least in the case $\varphi\equiv 0$ considered in that paper. As already recalled, no local existence result is available instead for smooth solutions. In the case of a general $\Phi(H_s)$, there is also no result on the invariance of convexity, since the result in \cite{CNRuffini} does not apply. In the classical case, convexity is preserved under some additional structural hypotheses on $\Phi$, see \cite{BS,AWei}. It is likely that a similar result should hold in the fractional case. We will not address this issue here and we will assume instead apriori the existence of a convex smooth solution.

The aim of this section is to prove that Theorem \ref{main} holds also for the more general problem \eqref{non-homo}. We first have the following lemma.
We denote, as before, by $E_t$ the set enclosed by $\mathcal M_t$.
\begin{lem}
Flow \eqref{non-homo} keeps the volume of $E_t$ constant and decreases its fractional perimeter $\text{Per}_s(E_t)$.
\end{lem}
\begin{proof}
The first part of the statement is an easy consequence of the choice of $\varphi(t)$. The second part follows exactly as in the proof of Lemma 3.1 in \cite{BS} in the local case.
\end{proof}

The uniform bounds on inner and outer radii and the lower bound for $H_s$ are obtained exactly as for the $\Phi(H_s)=H_s$ case (see Section~\ref{SE:3}), since they just rely on convexity and on the fact that the flow preserves volume and decreases the $s$-perimeter. Hence, we immediately have the following
\begin{prop}\label {lower-bound}
Let $F:\mm_0 \times [0,T) \to \rl^n$, with $0 <T \leq +\infty$, be a smooth convex solution of \eqref{non-homo}. Then there exist positive constants $0 <R_1 \leq R_2$, only depending on $E_0$, such that
$$
R_1 \leq \underline{\rho_{E_t}} \leq \overline{\rho_{E_t}} \leq R_2, \qquad \forall t \in [0,T).
$$
In addition, there exists $K_1>0$ such that $H_s(p,t) \geq K_1$ for all $(p,t) \in \mm_0 \times [0,T)$.
\end{prop}

{F}rom the previous proposition, we deduce again that, by choosing a smaller radius, we can find a ball with fixed center which remains inside $E_t$ for a time interval with fixed length, that is Lemma \ref{sphinside} holds also for solutions of the nonlinear flow \eqref{non-homo}. The proof of this fact is again by a comparison argument and we refer to Lemma 3.6 in \cite{BS} for the details.

In order to prove the analogue of Theorem \ref{main} for the flow \eqref{non-homo}, it remains to establish the upper bound on $H_s$.

\begin{prop}\label{main2}
Let $F:\mm_0 \times [0,T) \to \rl^n$, with $0 <T \leq +\infty$, be a smooth convex solution of \eqref{non-homo}. We have that, at any time $t\in [0,T)$ 
$$\Phi(H_s)\le K_3,$$
where $K_3$ is a positive constant depending only on $n,\,s,$ and $E_0$.
\end{prop}
\begin{proof}
The proof is similar to the one given in Section~\ref{S5}. We show in detail how the argument is adapted to the case of a general speed, for the sake of clarity.
We consider again the support function
$$u(p,t)=\langle F(p,t), \nu(p,t)\rangle.$$

If now $F$ evolves according to \eqref{non-homo}, recalling Lemma \ref{evol-cl}-(iii) and the expression for $\nabla^{\mathcal M} H_s$,
we have 
\begin{equation}\label{u_t}
\begin{split}
\partial_t u&=\langle \partial_tF, \nu\rangle+ \langle F, \partial_t\nu\rangle\\
&=-\Phi(H_s)+\varphi(t)+\Phi'(H_s)\langle F, \nabla^{\mathcal M} H_s\rangle\\
&=-\Phi(H_s)+\varphi(t)+2s(1-s)\Phi'(H_s)\int_{\mathcal M_t}\frac{x^T\cdot \nu(y)}{|x-y|^{n+s}}d\mu(y).
\end{split}
\end{equation}

Moreover, using Lemma \ref{evol-cl}-(ii), we have that the fractional mean curvature
satisfies
\begin{equation}\label{H_s-t}
\frac{\partial_t H_s(x)}{2s(1-s)}=\int_{\mathcal M_t}\frac{\Phi(H_s(y))-\Phi(H_s(x))}{|x-y|^{n+s}}d\mu(y) + (\Phi(H_s(x)-\varphi(t))\int_{\mathcal M_t}\frac{1-\nu(x)\cdot \nu(y)}{|x-y|^{n+s}}d\mu(y).\end{equation}
We define, similarly as before, but with the new velocity $\Phi(H_s)$,
$$W=\frac{\Phi(H_s)}{u(x)-\alpha},$$
where $\alpha$ is chosen in the same way as in Section~\ref{S5}. We have that
\begin{equation}\label{W1}
\begin{split}
\frac{\partial_ t W}{2s(1-s)}&=\frac{1}{2s(1-s)} \left[ \frac{\Phi'(H_s)\partial_t H_s}{u(x)-\alpha}-\frac{\Phi(H_s)\partial_ tu}{(u(x)-\alpha)^2} \right]\\
&=\frac{\Phi'(H_s(x))}{u(x)-\alpha}\left[\int_{\mathcal M_t}\frac{\Phi(H_s(y))-\Phi(H_s(x))}{|x-y|^{n+s}}d\mu(y)
 \right.\\
& \hspace{1em}+\left. (\Phi(H_s(x))-\varphi(t))\int_{\mathcal M_t}\frac{1-\nu(x)\cdot \nu(y)}{|x-y|^{n+s}}d\mu(y)\right]\\
&\hspace{1em}- \frac{\Phi(H_s(x))}{(u(x)-\alpha)^2}\left[\frac{-\Phi(H_s(x))+\varphi(t)}{2s(1-s)}+\Phi'(H_s(x))\int_{\mathcal M_t}\frac{x^T\cdot \nu(y)}{|x-y|^{n+s}}d\mu(y) \right]\\
&<\frac{\Phi'(H_s(x))}{u(x)-\alpha}\left[\int_{\mathcal M_t}\frac{\Phi(H_s(y))-\Phi(H_s(x))}{|x-y|^{n+s}}d\mu(y)
+ \Phi(H_s(x))\int_{\mathcal M_t}\frac{1-\nu(x)\cdot \nu(y)}{|x-y|^{n+s}}d\mu(y)\right]\\
&\hspace{1em}- \frac{\Phi(H_s(x))}{(u(x)-\alpha)^2}\left[\frac{-\Phi(H_s(x))}{2s(1-s)}+\Phi'(H_s(x))\int_{\mathcal M_t}\frac{x^T\cdot \nu(y)}{|x-y|^{n+s}}d\mu(y) \right].
\end{split}
\end{equation}

By the definition of $W$ we have that
\begin{equation}\label{A}
\begin{split}
&\int_{\mathcal M_t}\frac{\Phi(H_s(y))-\Phi(H_s(x))}{|x-y|^{n+s}}d\mu(y)\\
&\hspace{1em}=\int_{\mathcal M_t} (u(y)-\alpha)\frac{(W(y)-W(x))}{|x-y|^{n+s}}d\mu(y) + W(x)\int_{\mathcal M_t}\frac{u(y)-u(x)}{|x-y|^{n+s}}d\mu(y).
\end{split}
\end{equation}

Moreover, formula \eqref{eterzo} holds unchanged, since it is independent on the velocity:
\begin{equation}\label{B}
\begin{split}
&\int_{\mathcal M_t}\frac{u(y)-u(x)}{|x-y|^{n+s}}d\mu(y) \\
&\hspace{1em}  =\frac{1}{2(1-s)}H_s(x)+\int_{\mathcal M_t}\frac{x^T\cdot \nu(y)}{|x-y|^{n+s}} d\mu(y)
- u(x)\int_{\mathcal M_t}\frac{1-\nu(x)\cdot \nu(y)}{|x-y|^{n+s}}d\mu(y).
\end{split}
\end{equation}

We combine now \eqref{A} and \eqref{B} to get
\begin{equation}\label{AB}
\begin{split}
&\frac{1}{u(x)-\alpha}\int_{\mathcal M_t}\frac{\Phi(H_s(y))-\Phi(H_s(x))}{|x-y|^{n+s}}d\mu(y)\\
&\hspace{1em} =\frac{1}{u(x)-\alpha}\int_{\mathcal M_t} (u(y)-\alpha)\frac{W(y)-W(x)}{|x-y|^{n+s}}d\mu(y) \\
&\hspace{2em} + \frac{W(x)}{u(x)-\alpha}\left[\frac{1}{2(1-s)}H_s(x)+ \int_{\mathcal M_t}\frac{x^T\cdot \nu(y)}{|x-y|^{n+s}}d\mu(y)-u(x)\int_{\mathcal M_t}\frac{1-\nu(x)\cdot \nu(y)}{|x-y|^{n+s}} d\mu(y)\right]\\
&\hspace{1em} =\frac{1}{u(x)-\alpha}\int_{\mathcal M_t} (u(y)-\alpha)\frac{W(y)-W(x)}{|x-y|^{n+s}}d\mu(y) + \frac{H_s(x)\Phi(H_s(x))}{2(1-s)(u(x)-\alpha)^2} \\
&\hspace{2em} + \frac{W(x)}{u(x)-\alpha} \int_{\mathcal M_t}\frac{x^T\cdot \nu(y)}{|x-y|^{n+s}}d\mu(y)-\frac{W(x)}{u(x)-\alpha} u(x)\int_{\mathcal M_t}\frac{1-\nu(x)\cdot \nu(y)}{|x-y|^{n+s}} d\mu(y).  
\end{split}
\end{equation}

Finally, plugging \eqref{AB} into \eqref{W1}, we obtain
\begin{equation}\label{final}
\begin{split}
\frac{\partial_t W}{2s(1-s)}&<\Phi'(H_s)\left[\frac{1}{u(x)-\alpha}\int_{\mathcal M_t}\frac{(W(y)-W(x))(u(y)-\alpha)}{|x-y|^{n+s}}d\mu(y) \right]\\
&\hspace{1em}   \frac{W^2}{2s(1-s)}+ \Phi'(H_s)W\left[ \frac{H_s(x)}{2(1-s)(u(x)-\alpha)} - 
\frac{\alpha}{u(x)-\alpha}\int_{\mathcal M_t}\frac{1-\nu(x)\cdot \nu(y)}{|x-y|^{n+s}}d\mu(y)\right].
\end{split}
\end{equation}
This inequality is the analogue of estimate \eqref{W_t} in the presence of a nonlinear speed $\Phi$. Again, we use Proposition \ref{hausvogteiplatz} to bound the last term and we get
\begin{equation}
\begin{split}
\frac{\partial_t W}{2s(1-s)}&<\Phi'(H_s)\left[\frac{1}{u(x)-\alpha}\int_{\mathcal M_t}\frac{(W(y)-W(x))(u(y)-\alpha)}{|x-y|^{n+s}}d\mu(y) \right]\\
&\hspace{1em}   +C_1 W^2+ \frac{W\Phi'(H_s)H_s}{(u(x)-\alpha)}\left[ C_2 - 
C_3H_s\right].
\end{split}
\end{equation}
Setting $\tilde W(t)=\sup_{\mathcal M_t} W(x,t)$, we have
$$
\partial_t \tilde W(t)\le C_1 \tilde W^2+\frac{\tilde W \Phi'(H_s) H_s}{u-\alpha}\left[ C_2 - 
C_3H_s\right],
$$
where $H_s=H_s(\tilde x,t)$ for a suitable $\tilde x$  such that $W(\tilde x,t)=\tilde W(t)$.

We choose now $K>3C_2/C_3$, so that $H_s\ge K$ implies $C_2-C_3H_s\le -\frac{2}{3}C_3H_s$. Suppose now that there exists $t^*$ such that $\tilde W(t^*)\ge \frac{\Phi(K)}{\alpha}$. Recalling that $u-\alpha \ge \alpha$ and using the monotonicity of $\Phi$, we deduce that $H_s(x^*,t^*) \ge K$ for any $x^*$ such that $ W(x^*,t^*)=\tilde W(t^*)$.  Hence, at $t=t^*$, we have
$$
\partial_t \tilde W\le C_1 \tilde W^2-\frac{2C_3\tilde W \Phi'(H_s) H_s^2}{3(u-\alpha)}\le \tilde W^2 \left[C_1-\frac{2C_3}{3}\frac{\Phi'(H_s)H_s^2}{\Phi(H_s)}\right].
$$
By property iii) of $\Phi$, we can choose $K$ large enough so that if $ H_s\ge K$ we have
$$C_1-\frac{2C_3}{3}\frac{\Phi'(H_s)H_s^2}{\Phi(H_s)}<-1,$$
which gives 
$$\partial_t \tilde W\leq -\tilde W^2.$$

From this last estimate, the conclusion follows by a comparison argument, exactly as in the proof of Proposition 3.7 in \cite{BS}.
\end{proof}

As a consequence of the boundedness of the speed $\Phi(H_s)$ and of property i) satisfied by $\Phi$, and recalling Proposition \ref{lower-bound}, we deduce the following

\begin{cor}\label{upper-bound-H}
We have that $H_s$ is uniformly bounded in $(0,T)$.
\end{cor}

\section{Convergence to a sphere}\label{S7}
In this section we prove our convergence result (Theorem \ref{sphere} for the case $\Phi(H_s)=H_s)$, that for the general problem \eqref{non-homo} reads as follows
\begin{thm}\label{sphere-Phi}
Let $F:\mm_0 \times [0,T) \to \rl^n$, with $0 <T \leq +\infty$, be a smooth convex solution of \eqref{non-homo} of class $C^{2,\beta}$ for some $\beta>s$ which satisfies property {\bf (R) }. Then $T=+\infty$, and $\mm_t$  converges to a round sphere as $t \to +\infty$ in $C^{2,\beta}$ norm, possibly up to translations.
\end{thm}

We first observe that, by the lower and upper bounds on $H_s$, we have that $\Phi'(H_s)$ is bounded from above and below by positive constants for every $t\in [0, +\infty)$.

The crucial step in the proof of Theorem \ref{sphere-Phi} is the following result.
\begin{prop}\label{convergence}
Under our assumption, we have that
$$\lim_{t \rightarrow +\infty} \max_{\mm_t}|\Phi(H_s(x)-\varphi(t)|=0.$$
\end{prop}

\begin{proof}
The proof follows the one in \cite{BS}, Proposition 4.4.
For any $t$, let $\overline H_s(t)$ be such that $\Phi(\overline H_s(t))=\varphi(t)$. Then, recalling \eqref{eqper}, we have
\begin{equation*}
\begin{split}
\frac{d}{dt}\text{Per}_s(E_t)&=\int_{\mm_t}H_s\varphi\, d\mu - \int_{\mm_t}H_s\Phi(H_s)\,d\mu\\
&= \int_{\mm_t}(H_s-\overline H_s)(\Phi(\overline H_s)-\Phi(H_s))\,d\mu\\
&=-\int_{\mm_t} |H_s-\overline H_s||\Phi(\overline H_s)-\Phi(H_s)|\,d\mu.
\end{split}
\end{equation*}

Hence, using the boundedness of $\Phi'$, we deduce that

\begin{equation*}
\frac{d}{dt}\text{Per}_s(E_t)\le -\frac{1}{\sup \Phi'}\int_{\mm_t}|\Phi(\overline H_s)-\Phi(H_s)|^2\, d\mu =
-\frac{1}{\sup \Phi'}\int_{\mm_t}|\Phi(\overline H_s)-\varphi|^2\, d\mu.
\end{equation*}

Suppose now, by contradiction, that there exists $\epsilon>0$ such that $|\Phi(H_s)-\varphi|=\epsilon$ at some point $(\bar p, \bar t)$.
By our regularity assumption and using Theorem \ref{dH}, we have that $H_s$ is uniformly Lipschitz, therefore there exists a uniform radius $r(\epsilon)>0$ for which
$$|\Phi(H_s)-\varphi|>\frac{\epsilon}{2}\quad \mbox{in}\,\, B((\bar p, \bar t), r(\epsilon)),$$
which implies
$$
\frac{d}{dt}\mbox{Per}_s(E_t)\le - \eta(\epsilon) \quad \mbox{for any }\, t \in [\bar t- r(\epsilon), \bar t+ r(\epsilon)],
$$
for some $\eta >0$.
The fact that $\mbox{Per}_s(E_t)>0$ and decreasing in time implies that the above property  cannot hold for $\bar t$ arbitrarily large. This shows that $|\Phi(H_s)-\varphi|$ tends to zero uniformly.
\end{proof}

We are now ready to give the proof of our convergence result.

\begin{proof}[Proof of Theorem \ref{sphere-Phi}]
Using our regularity assumption {\bf (R) } and the uniform bounds for $H_s$ of Corollary \ref{upper-bound-H} and Theorem \ref{main}, we deduce that the flow exists for all $t \in [0,\infty)$ and that the hypersurfaces  $\mm_t$, possibly up to translations, 
are bounded in the $C^{2,\beta}$ norm uniformly in $t$. Hence, the $\mm_t$ are precompact in $C^{2,\beta'}$ for $\beta'<\beta$. By Proposition \ref{convergence} and the stability results of \cite{Cozzi}, we have that any possible subsequential limit as $t \to +\infty$ has constant fractional curvature. Then Theorem~\ref{Alex} ensures that the limit is a ball, with radius uniquely determined by the volume constraint. The uniqueness of the  subsequential limit easily implies that the whole family $\mm_t$ converges to a sphere as $t \to +\infty$.
\end{proof}

\bibliographystyle{alpha}
\newcommand{\etalchar}[1]{$^{#1}$}

\end{document}